\documentclass[11pt]{article}
\usepackage[arrow, matrix]{xy}
\usepackage{amsmath}
\usepackage{cancel}
\usepackage{amssymb}
\usepackage{amsthm}
\usepackage[mathscr]{eucal}
\input{epsf}
\theoremstyle{definition}
\newtheorem{definition}{Definition}
\newtheorem{theorem}[definition]{Theorem}
\newtheorem{proposition}[definition]{Proposition}
\newtheorem{lemma}[definition]{Lemma}
\newtheorem{corollary}[definition]{Corollary}
\theoremstyle{remark}

\newcounter{enumctr}

\newcommand{\R}{\mathbb{R}}

\newcommand{\mP}{\mathbb{P}}

\newcommand{\cF}{\mathcal{F}}

\renewcommand{\phi}{\varphi}

\begin{document}
\title{\vspace*{-10mm}
An analytical proof for synchronization of
stochastic phase oscillator}
\author{
Yuzuru Sato\footnote{Email:  ysato@math.sci.hokudai.ac.jp, RIES / Department of Mathematics, Hokkaido University, Kita 20 Nishi 10, Kita-ku, Sapporo Hokkaido 001-0020, Japan}, Doan Thai Son\footnote{Email: dtson@math.ac.vn, Institute of Mathematics, Vietnam Academy of Science and Technology, 18 Hoang Quoc Viet Road, Cay Giay, Ha Noi, Vietnam},\;
Nguyen Thi The\footnote{Email: thent@vinhuni.edu.vn, Vinh University, 182 Le Duan, Vinh city, Nghe An province, Vietnam}, \; and Hoang The Tuan\footnote{Email: httuan@math.ac.vn, Institute of Mathematics, Vietnam Academy of Science and Technology, 18 Hoang Quoc Viet Road, Cay Giay, Ha Noi, Vietnam}
}
\date{}
\maketitle
\begin{abstract}
In this paper, we show that under a generic condition of the coefficient of a stochastic phase oscillator the Lyapunov exponent of the linearization along an arbitrary solution is always negative. Consequently, the generated random dynamical system exhibits a synchronization.
\end{abstract}
\section{Introduction}
Synchronization of oscillators under the presence of noise has been studied in the literature of Kuramoto theory of coupled limit cycles \cite{Kuramoto}. The noise-induced synchronization is first studied for pulse-noised oscillators whose phase response function is sinusoidal \cite{pikovsky84}. The phase reduction of limit cycle oscillators coupled through common white Gaussean noise is proposed as a model of reliability of dynamics of spiking neurons, and its noise-induced phase synchronization is studied \cite{TT_04}. The phase reduction for stochastic oscillators are verified in the case of random telegraphic and pulse noise \cite{nagai05, nakao05}.

When noise is strong, desynchronization can be observed in stochastic coupled limit cycles \cite{goldobin05,goldobin05-2}. Synchronization with colored noise is also studied \cite{teramae06,kurebayashi12}. After that, the phase reduction of stochastic differential equation with limit cycles has been discussed \cite{nakao07, yoshimura08, teramae09, goldobin10} and it has been confirmed that for a stochastic phase oscillator involving small noise intensity, the Lyapunov exponent is negative. These results are mainly based on the uniformity of the invariant density.

In this paper, starting with phase oscillator equations, without the assumption of small noise intensity, and for almost arbitrary phase response functions, we are still able to show that the negativity of Lyapunov exponent of uncoupled limit cycle oscillator with common noise.
Our result implies that we always have synchronization in uncoupled limit cycles
with common noise if stochastic phase reduction is possible.

The paper is organized as follows: Section \ref{Section2} consists of two parts. In the first part, we introduce a model of stochastic phase oscillator and some fundamental concepts. The second part is devoted to state the main result of the paper. The proof of the main result is given in Section \ref{Section3}.
\section{Preliminaries and main results}\label{Section2}
\subsection{Stochastic phase oscillator}\label{Subsection2.1}
Consider the following model of Stratonovich stochastic phase oscillator
\begin{equation}\label{Circle}
d \phi_t= \rho\;dt+ f(\phi_t) \circ dW_t\qquad\mbox{mod } 2\pi,
\end{equation}
where $\rho>0$ and $f:[0,2\pi]\rightarrow \R$ is a smooth $2\pi$-periodic function.  An equivalent equation of \eqref{Circle} in Ito stochastic differential equation is given by
\begin{equation}\label{Circle02}
d\phi_t= \left(\rho+ \frac{1}{2}f^{'}(\phi_t) f(\phi_t)\right)\,dt+ f(\phi_t)\,dW_t.
\end{equation}
To study the sample-path properties of solutions of \eqref{Circle}, it is convenient to work with the canonical space of Brownian motion, see e.g. \cite[Appendix A2]{Arnold}. Precisely, let $\Omega=C_0(\R_{\geq 0},\R)$ denote the set of all continuous functions $\omega:\R_{\geq 0}\rightarrow \R$ satisfying that $\omega(0)=0$. Let $\Omega$ be equipped with the following metric
\[
\kappa(\omega,\widehat\omega)
:=
\sum_{n=1}^{\infty}
\frac{1}{2^n}\frac{\|\omega-\widehat\omega\|_n}{1+\|\omega-\widehat\omega\|_n},\qquad \|\omega-\widehat\omega\|_n:=\sup_{t\leq 0\leq n} |\omega(t)-\widehat\omega(t)|.
\]
Denote by
   $\mathcal{F}$  the Borel
$\sigma $-algebra of  $(\Omega,\kappa)$ and $\mP$ the Wiener measure  on $(\Omega,\mathcal F) $ such that $W_t(\omega) := \omega(t)$ is a Brownian motion on $(\Omega,\cF,\mP)$.  Define the shift transformation $\theta_t: \Omega  \rightarrow  \Omega$, where $t \in \R$, by
$$
\theta_t\omega(s):=\omega(t+s)-\omega(t)\qquad\hbox{for all } s \in \R_{\geq 0}.
$$
It is well known that the solutions of \eqref{Circle} give rise to a random dynamical system $\phi:\R_{\geq 0}\times \Omega\times [0,2\pi]\rightarrow [0,2\pi]$ over the metric dynamical system $(\Omega,\cF,\mP,(\theta_t)_{t\in\R})$, i.e.
\[
\phi(t+s,\omega,x)=\phi(t,\theta_s\omega,\phi(s,\omega,x))\quad \hbox{for } t,s\in\R_{\geq 0},
\]
see \cite[Theorem 2.3.32]{Arnold}. Note that
 $$
 \left(\rho+ \frac{1}{2}f^{'}(\phi) f(\phi)\right)^2 + f(\phi)^2  \neq 0\qquad\hbox{for all } \phi \in [0,2\pi],
 $$
which indicates that  the coefficients of \eqref{Circle} satisfy the H\"{o}mander condition. Consequently, there exists a unique smooth stationary distribution $\rho$ for the one-point motion $\phi(t,\omega)$. Furthermore, the density $p_{\rm st} $ of $\rho$ satisfies the following Fokker-Planck equation
 \begin{equation}\label{FK_Eq}
	\frac{\partial p(t,\phi)}{\partial t}
	=
	-\frac{\partial }{\partial \phi}\left((\rho+\frac{1}{2}f^{'}(\phi)f(\phi)) p(t,\phi)\right)
	+
	\frac{1}{2}
	\frac{\partial^2 }{\partial \phi^2}\left(f(\phi)^2 p(t,\phi)\right).
 \end{equation}
Note that the skew product flow $(\Theta_t)_{t\in\R}$ on $\Omega\times [0,2\pi] \rightarrow \Omega\times [0,2\pi]$ defined by
\[
\Theta_t(\omega,x):=(\theta_t\omega,\phi(t,\omega,x))
\]
is an ergodic flow preserving the product probability  $\mP\times \rho$. Then, by using Birkhoff's erogidic theorem, the Lyapunov exponent of the linearization of $\phi(t,\omega,x)$
\begin{equation}\label{LE}
\lambda:=\lim_{t\to\infty}\log|D\phi(t,\omega,x)|
\end{equation}
for $\mP\times \rho-(\omega,x)$ exists and is constant.
\subsection{Statements of the main results}\label{Subsection2.2}
In the remaining of the paper, the function $f$ is assumed to fulfill the following generic property:
\begin{itemize}
\item [(H1)] The function $f$ is not a constant function.
\item [(H2)] The graph of the function $f$ and the zero function intersect transversally, i.e. $f^{'}(\phi)\not=0$ for all $\phi\in \mathcal N(f)$, where
\begin{equation}\label{Nullset}
\mathcal N(f):=\{\phi\in [0,2\pi]: f(\phi)=0\}.
\end{equation}
\end{itemize}
The first part of the main result of this paper indicates that the Lyapunov exponent of the linearization of the random dynamical system generated by \eqref{Circle} is negative.
\begin{theorem}[Negativity of Lyapunov exponents of generic stochastic phase oscillators]\label{MainThm1} Suppose that the assumptions (H1) and (H2)  hold. Then, the Lyapunov exponent of the linearization of the random dynamical system $\phi$ generated by \eqref{Circle} defined as in \eqref{LE} is strictly negative.
\end{theorem}
Next, we formulate in the following corollary a result on  synchronization of the generated random dynamical systems from several point of views.
\begin{corollary}[Synchronization of generic stochastic phase oscillators]
Suppose that the assumptions (H1) and (H2) hold. Then, the following statements hold:
\begin{itemize}
\item [(i)] For any $x,y\in [0,2\pi]$ we have
\[
\lim_{t\to\infty} d(\phi(t,\omega,x),\phi(t,\omega,y))=0\qquad\hbox{for } \mP-\hbox{a.e. } \omega\in\Omega.
\]
\item [(ii)] There exist a random fixed point $a:\Omega\rightarrow [0,2\pi]$ such that for all $x\in [0,2\pi]$
\[
\lim_{t\to\infty} d(\phi(t,\theta_{-t}\omega,x), a(\omega))=0\qquad\hbox{for } \mP-\hbox{a.e. } \omega\in\Omega.
\]
\end{itemize}
\end{corollary}
\begin{proof}
See \cite{Baxendale}.
\end{proof}
\section{Proof of the main result}\label{Section3}
Before going to the proof of the main result, we need the following preparatory materials. Note that the second part of the following proposition is named as Furtensberg-Khaminskii formula for the Lyapunov exponent of nonlinear stochastic differential equations, see \cite[Subsection 6.2.2]{Arnold}. To make the paper self-contained, we give a short proof of this formula.
\begin{proposition}\label{Fundamental}
	\begin{itemize}
		\item [(i)]		
		The density $p_{\rm st}(\phi)$ of the unique stationary measure of the Markov-process generated by \eqref{Circle02} satisfies the following differential equation
		\begin{equation}\label{FK}
		f(\phi)^2 \;p_{\rm st}^{'}(\phi)= \left(
		2\rho- f^{'}(\phi)f(\phi)
		\right) p_{\rm st}(\phi)+
		C,
		\end{equation}
		where $C$ is a constant.
		\item[(ii)] The Lyapunov exponent of the linearization of  random dynamical system generated by \eqref{Circle02} is given by
		\[
		\lambda=\frac{1}{2}\int_0^{2\pi} f^{''}(\phi) f(\phi) p_{\rm st}(\phi)\;d\phi.
		\]
	\end{itemize}
\end{proposition}
\begin{proof}
	(i) Using \eqref{FK_Eq}, the density function $p_{\rm st} (\phi)$ of the unique stationary measure of \eqref{Circle02} satisfies the following equality
	\[
	-\frac{\partial }{\partial \phi}\left((\rho+\frac{1}{2}f^{'}(\phi)f(\phi)) p_{\rm st}(\phi)\right)
	+
	\frac{1}{2}
	\frac{\partial^2 }{\partial \phi^2}\left(f(\phi)^2 p_{\rm st}(\phi)\right)=0.
	\]
	Consequently, there exists a constant $C$ such that
	\[
	-\left(2\rho+f^{'}(\phi)f(\phi)\right) p_{\rm st}(\phi)+
	\frac{\partial }{\partial \phi}\left(f(\phi)^2 p_{\rm st}(\phi)\right)=C.
	\]
Expanding the term $\frac{\partial }{\partial \phi}\left(f(\phi)^2 p_{\rm st}(\phi)\right)$ in the preceding equality completes the proof of this part.

	\noindent
	(ii) The linearization along a fixed solution $\phi_t$ of \eqref{Circle} is given by
	\[
	dv_t= f^{'}(\phi_t)v_t\circ dW_t.
	\]
	Define $r_t:=\log|v_t|$. Then, the equation for $r_t$ is given by
	\[
	dr_t= f^{'}(\phi_t)\circ dW_t.
	\]
	The Ito form of this equation, see e.g. \cite[pp. 137-138]{Kloeden},  is
	\[
	dr_t= \frac{1}{2}f^{''}(\phi_t)f(\phi_t)\;dt+ f^{'}(\phi_t) dW_t.
	\]
	Therefore, the Lyapunov exponent $\lambda=\lim_{t\to\infty}\frac{1}{t} r_t$ can be computed as follows
	\begin{eqnarray*}
	\lambda
	&=& \frac{1}{2}\lim_{t\to\infty}\frac{1}{t} \int_0^tf^{''}(\phi_s)f(\phi_s)\;ds
	+\lim_{t\to\infty}\frac{1}{t} \int_0^t f^{'}(\phi_s)\,dW_s\\
	&=&
	\frac{1}{2} \int_0^{2\pi} f^{''}(\phi)f(\phi) p_{\rm st}(\phi)\;d\phi.
	\end{eqnarray*}
The proof is complete.
\end{proof}
Using the preceding proposition, we show in the following lemma that under the assumption (H1), the stationary distribution of \eqref{Circle} is not uniform on $[0,2\pi]$.
\begin{lemma}\label{Non-constant}
Suppose that (H1) holds. Then, the density function $p_{\rm st}$ is not constant.
\end{lemma}
\begin{proof}
Suppose a contrary, i.e. $p_{\rm st}(\phi)=\frac{1}{2\pi}$ for all $\phi\in [0,2\pi]$. Then, by \eqref{FK} we arrive at
\[
f^{'}(\phi)f(\phi)=2\rho + 2\pi C\qquad\hbox{for } \phi\in [0,2\pi].
\]
Since $(f(\phi)^2)^{'}=2f^{'}(\phi)f(\phi)$ it follows that
\[
f(\phi)^2=f(0)^2 + (\rho+\pi C)\phi,
\]
which together with $2\pi$-periodicity of $f$ implies that $\rho+\pi C=0$. Thus, $f(\phi)^2$ is a constant function and this contradicts to (H1). The proof is complete.
\end{proof}
%
%
Note that by continuity of the function $f$, the set $\mathcal N(f)$ defined as in \eqref{Nullset} is a closed subset of $[0,2\pi]$. Furthermore, the assumption (H2) implies that the set  $\mathcal N(f)$ has no accumulation point. Consequently, the set $\mathcal N(f)$ is either empty or finite.  In what follows, we separate the proof of Theorem \ref{MainThm1} into two cases:
\begin{itemize}
\item \emph{Non-vanishing noise}: The set $\mathcal N(f)$ is empty.
\item \emph{Vanishing noise}: The set $\mathcal N(f)$ is not empty and finite.
\end{itemize}
Before going to the proof of the theorem for the non-vanishing noise, we need the following technical lemma.
\begin{lemma}\label{Lemma1}
Let $g:[0,2\pi]\rightarrow \R_{>0}$ be an arbitrary continuous function. Then,
\begin{equation}\label{Inequality}
\int_0^{2\pi} g(\phi) p_{\rm st}(\phi)\;d\phi
\int_0^{2\pi} \frac{g(\phi)}{p_{\rm st}(\phi)}\;d\phi
> \left(\int_0^{2\pi} g(\phi)\;d\phi\right)^2.
\end{equation}
\end{lemma}
\begin{proof}
Let
\[
u(\phi):=\sqrt{g(\phi) p_{\rm st}(\phi)}\quad\hbox{and } \quad v(\phi):=\sqrt{\frac{g(\phi)}{p_{\rm st}(\phi)}}.
\]
Using the H\"{o}lder inequality, we obtain
\begin{equation}\label{Holder}
\int_0^{2\phi} u(\phi)^2\;d\phi \int_0^{2\phi} v(\phi)^2\;d\phi
\geq
\left(\int_0^{2\pi} u(\phi)v(\phi)\;d\phi\right)^2.
\end{equation}
 Furthermore, the equality holds iff $p_{\rm st}(\cdot)$ is constant and using Lemma \ref{Non-constant}, the equality of \eqref{Holder} cannot hold.  The proof is complete.
\end{proof}
\begin{proof} [Proof of Theorem \ref{MainThm1} for non-vanishing noise] In such a case, $f(\phi)\not=0$ for all $\phi\in [0,2\pi]$. Thus, the Fokker-Planck equation in \eqref{FK} becomes an ordinary differential equation of the form
\begin{equation}\label{FK_02}
p_{\rm st}^{'}(\phi)= \left(
		\frac{2\rho}{f(\phi)^2}- \frac{f^{'}(\phi)}{f(\phi)}
		\right) p_{\rm st}(\phi)+
		\frac{C}{f(\phi)^2},
\end{equation}
which implies that
\begin{equation*}\label{FK_03}
f^{'}(\phi)f(\phi) p_{\rm st}^{'}(\phi)
=
2\rho\frac{f^{'}(\phi)}{f(\phi)} p_{\rm st}(\phi)- f^{'}(\phi)^2p_{\rm st}(\phi)+C \frac{f^{'}(\phi)}{f(\phi)}.
\end{equation*}
Consequently,
\begin{eqnarray*}
\int_0^{2\pi} f^{'}(\phi)f(\phi) p_{\rm st}^{'}(\phi)\;d\phi
&=&
2\rho
\int_0^{2\pi}\frac{f^{'}(\phi)}{f(\phi)}p_{\rm st}(\phi)\;d\phi -\int_{0}^{2\pi} f^{'}(\phi)^2 p_{\rm st}(\phi)\;d\phi\\
&&+ C\int_0^{2\pi} \frac{f^{'}(\phi)}{f(\phi)}\;d\phi.
\end{eqnarray*}
Since $\int_0^{2\pi} \frac{f^{'}(\phi)}{f(\phi)}\;d\phi=\ln f(2\pi)-\ln f(0)=0$ it follows that
\[
\int_0^{2\pi} f^{'}(\phi)f(\phi) p_{\rm st}^{'}(\phi)\;d\phi
+
\int_{0}^{2\pi} f^{'}(\phi)^2 p_{\rm st}(\phi)\;d\phi
=2\rho \int_0^{2\pi}\frac{f^{'}(\phi)}{f(\phi)}p_{\rm st}(\phi)\;d\phi.
\]
Thus, from Proposition \ref{Fundamental}(ii) and the integration by parts formula we arrive at
\begin{eqnarray}
\lambda
&=&
\frac{1}{2}\int_0^{2\pi} f^{''}(\phi) f(\phi) p_{\rm st}(\phi)\;d\phi\notag\\
&=&
-\frac{1}{2}
\left(
\int_0^{2\pi} f^{'}(\phi)f(\phi) p_{\rm st}^{'}(\phi)\;d\phi
+
\int_{0}^{2\pi} f^{'}(\phi)^2 p_{\rm st}(\phi)\;d\phi
\right)\notag
\\
&=&
-\rho \int_0^{2\pi}\frac{f^{'}(\phi)}{f(\phi)}p_{\rm st}(\phi)\;d\phi.
\label{FK_04}
\end{eqnarray}
On the other hand, taking the integral of both sides of \eqref{FK_02} yields that
\[
\int_0^{2\pi}
\frac{f^{'}(\phi)}{f(\phi)}p_{\rm st}(\phi)\;d\phi
=
2\rho\int_0^{2\pi}
\frac{p_{\rm st}(\phi)}{f(\phi)^2}\;d\phi
+
C \int_0^{2\pi} \frac{1}{f(\phi)^2}\;d\phi,
\]
which together with \eqref{FK_04} implies that
\begin{equation}\label{FK_05}
-\frac{\lambda}{\rho}
=
2\rho\int_0^{2\pi}
\frac{p_{\rm st}(\phi)}{f(\phi)^2}\;d\phi
+
C \int_0^{2\pi} \frac{1}{f(\phi)^2}\;d\phi.
\end{equation}
Next, we are applying Lemma \ref{Lemma1} to indicate that $\lambda<0$. For this purpose, we divide both sides of \eqref{FK_02} by the term $p_{\rm st}(\phi)$ to obtain
\[
\frac{p_{\rm st}^{'}(\phi)}{p_{\rm st}(\phi)}
=\frac{2\rho}{f(\phi)^2}-\frac{f^{'}(\phi)}{f(\phi)}+\frac{C}{f(\phi)^2 p_{\rm st}(\phi)}.
\]
Taking the integral of both sides of the preceding equality yields that
\begin{equation}\label{FK_06}
0=\int_0^{2\pi}
\frac{p_{\rm st}^{'}(\phi)}{p_{\rm st}(\phi)}
\;d\phi=
2\rho
\int_0^{2\pi}\frac{1}{f(\phi)^2}\;d\phi
+
C\int_{0}^{2\pi}\frac{1}{f(\phi)^2 p_{\rm st}(\phi)}\;d\phi.
\end{equation}
By virtue of Lemma \ref{Lemma1}, we have
\[
\int_{0}^{2\pi}\frac{1}{f(\phi)^2 p_{\rm st}(\phi)}\;d\phi
\int_{0}^{2\pi}\frac{p_{\rm st}(\phi)}{f(\phi)^2 }\;d\phi
>
\left(\int_0^{2\pi}\frac{1}{f(\phi)^2}\;d\phi\right)^2,
\]
which together with \eqref{FK_06} implies that
\[
\int_{0}^{2\pi}\frac{p_{\rm st}(\phi)}{f(\phi)^2 }\;d\phi
>
\frac{|C|}{2\rho} \int_0^{2\pi}\frac{1}{f(\phi)^2}\;d\phi.
\]
Consequently, from \eqref{FK_05} we derive that $\lambda<0$ and the assertion is proved in this case.
\end{proof}
The remaining of this paper is devoted to prove Theorem \ref{MainThm1} in the vanishing noise case. In comparison to the non-vanishing noise case, the difficulty is that the equation \eqref{FK} is an differential algebraic equation. To overcome this difficulty, we will treat \eqref{FK} on different domains of $\phi$ in which $f(\phi)$ is non-vanishing. For this purpose, we need the following preparatory lemma.
\begin{lemma}\label{Lemma2}
Let $\eta<\zeta$ be two consecutive elements of the set $\mathcal N(f)$. Then, the following statements hold:
\begin{itemize}
\item [(i)] There exists $\delta^*\in (0,\frac{\zeta-\eta}{2})$ and a smooth and strictly increasing function $r:[0,\delta^*)\rightarrow [0,\frac{\zeta-\eta}{2}]$ satisfying that $r(0)=0$ and
\begin{equation}\label{Requirement1}
f(\eta+\delta)=f(\zeta-r(\delta))\qquad\hbox{for all } \delta\in [0,\delta^*).
\end{equation}
\item [(ii)] The density function $p_{\rm st}$ is not constant on the interval $[\eta,\zeta]$.
\end{itemize}
\end{lemma}
\begin{proof}
(i) Since $\eta<\zeta$ are two consecutive elements of $\mathcal N(f)$ it follows that either $f(\phi)>0$ for all $\phi\in (\eta,\zeta)$ or $f(\phi)<0$ for all $\phi\in (\eta,\zeta)$. W.l.o.g. we assume that $f(\phi)>0 $  for all $\phi\in (\eta,\zeta)$. Using (H2), we obtain that
\begin{equation}\label{Sign_Eq1}
f'(\eta)>0>f'(\zeta).
\end{equation}
Thus, there exists $\widetilde\delta\in (0,\frac{\zeta-\eta}{2})$ such that the restriction function $f_1:=f|_{(\eta,\eta+\widetilde\delta)}: (\eta,\eta+\widetilde\delta)\rightarrow (f(\eta),f(\eta+\widetilde\delta))$ is strictly increasing and a smooth diffeomorphism. Similarly, there exists $\widehat \delta \in \frac{\zeta-\eta}{2})$  such that the restriction function $f_2:=f|_{(\zeta-\widehat\delta,\zeta)}: (\zeta-\widehat\delta,\zeta)\rightarrow (f(\zeta),f(\zeta-\widehat\delta))$ is strictly decreasing and a smooth diffeomorphism. Shrinking $\widetilde\delta$, if necessary, we can also choose $\widehat\delta$ such that $f(\zeta-\widehat\delta)=f(\eta+\delta)$. Let $\delta^*:=\widetilde\delta$ and define $r:[0,\delta^*)\rightarrow [0,\frac{\zeta-\eta}{2})$ by $r(0)=0$ and
\[
r(\delta):=\zeta-f_2^{-1}\circ f_1(\eta+\delta)\qquad\hbox{for all } \delta\in (0,\delta^*).
\]
Then, $r$ is strictly increasing and smooth function and satisfies \eqref{Requirement1}.

\noindent
(ii) Suppose a contrary that the function $p_{\rm st} (\phi)$ is constant on $[\eta,\zeta]$. This implies that $p^{'}_{\rm st} (\phi)=0$ for $\phi\in [\eta,\zeta]$. Then, by \eqref{FK} the function $f^{'}(\phi)f(\phi)$ is constant on $[\eta,\zeta]$. This together with the fact that $f(\eta)=0$ implies that $f(\phi)=0$ for all $\phi\in [\eta,\zeta]$. This leads to a contradiction and the proof is complete.
\end{proof}
\begin{proof} [Proof of Theorem \ref{MainThm1} for vanishing noise] In such a case, the set $\mathcal N(f)$ has a finite element, i.e.
\[
\mathcal N(f)=\{\phi_0,\phi_1,\dots, \phi_k\},\quad \hbox{where}\quad \phi_0<\phi_1<\dots<\phi_k.
\]
Replacing the function $f(\cdot)$ by $f(\cdot-\phi_0)$, if necessary, we can assume without affecting the proof that $\phi_0=0$ and therefore $\phi_k=2\pi$. Next, by Proposition \ref{Fundamental}, to show $\lambda<0$ it is sufficient to prove the following statement for all $i=0,\dots,k-1$
\begin{equation*}
I_i<0, \qquad \hbox{where } I_i:=\int_{\phi_i}^{\phi_{i+1}} f^{''}(\phi)f(\phi)  p_{\rm st}(\phi)\;d\phi.
\end{equation*}
For this purpose, we choose and fix $i\in\{0,\dots,k-1\}$. Using the integration by parts formula, we obtain
\begin{equation}\label{Eq_01}
I_i=
-\left(
\int_{\phi_i}^{\phi_{i+1}} f^{'}(\phi)^2 p_{\rm st}(\phi)\;d\phi+
\int_{\phi_i}^{\phi_{i+1}} f^{'}(\phi)f(\phi) p_{\rm st}^{'}(\phi)\;d\phi
\right).
\end{equation}
By virtue of Lemma \ref{Lemma2}, there exists $\delta^*\in (0,\frac{\zeta-\eta}{2})$ and a smooth function $r:[0,\delta^*)\rightarrow $ satisfying that $r(0)=0$ and
\begin{equation}\label{Eq_Functionr}
f(\phi_i+\delta)=f(\phi_{i+1}-r(\delta))\qquad\hbox{for all } \delta\in [0,\delta^*).
\end{equation}
By definition of $\phi_i,\phi_{i+1}$ and the set $\mathcal N(f)$, we have $f(\phi)\not=0$ for all $\phi\in (\phi_i,\phi_{i+1})$. Thus, from \eqref{FK} we have for all $\phi\in (\phi_i,\phi_{i+1})$
\begin{equation}\label{FK_02a}
p_{\rm st}^{'}(\phi)= \left(
		\frac{2\rho}{f(\phi)^2}- \frac{f^{'}(\phi)}{f(\phi)}
		\right) p_{\rm st}(\phi)+
		\frac{C}{f(\phi)^2},
\end{equation}
which implies that
\begin{equation*}\label{FK_03a}
f^{'}(\phi)f(\phi) p_{\rm st}^{'}(\phi)
=
2\rho\frac{f^{'}(\phi)}{f(\phi)} p_{\rm st}(\phi)- f^{'}(\phi)^2p_{\rm st}(\phi)+C \frac{f^{'}(\phi)}{f(\phi)}.
\end{equation*}
Consequently, for all $\delta\in (0,\delta^*)$ we have
\begin{eqnarray*}
&& \int_{\phi_i+\delta}^{\phi_{i+1}-r(\delta)} f^{'}(\phi)f(\phi) p_{\rm st}^{'}(\phi)+ f^{'}(\phi)^2 p_{\rm st}(\phi)\;d\phi\\
&=&
2\rho
\int_{\phi_i+\delta}^{\phi_{i+1}-r(\delta)}\frac{f^{'}(\phi)}{f(\phi)}p_{\rm st}(\phi)\;d\phi + C\int_{\phi_i+\delta}^{\phi_{i+1}-r(\delta)} \frac{f^{'}(\phi)}{f(\phi)}\;d\phi.
\end{eqnarray*}
From \eqref{Eq_Functionr}, we derive that
\[
\int_{\phi_i+\delta}^{\phi_{i+1}-r(\delta)} \frac{f^{'}(\phi)}{f(\phi)}\;d\phi=\ln|f(\phi_{i+1}-r(\delta))|-\ln|f(\phi_i+\delta)|=0.
\]
Thus, by \eqref{Eq_01}
\begin{equation}\label{Eq_05}
I_i=-2\rho \lim_{\delta\to 0^+}\int_{\phi_i+\delta}^{\phi_{i+1}-r(\delta)}  \frac{f^{'}(\phi)}{f(\phi)}p_{\rm st}(\phi)\;d\phi.
\end{equation}
Taking the integral of both sides of \eqref{FK_02a} from $\phi_i+\delta$ to $\phi_{i+1}-r(\delta)$ yields that
\begin{eqnarray*}
&& p_{\rm st}(\phi_{i+1}-r(\delta))-p_{\rm st}(\phi_{i}+\delta)
+
\int_{\phi_i+\delta}^{\phi_{i+1}-r(\delta)}  \frac{f^{'}(\phi)}{f(\phi)}p_{\rm st}(\phi)\;d\phi\\
&=&
2\rho\int_{\phi_i+\delta}^{\phi_{i+1}-r(\delta)} \frac{p_{\rm st}(\phi)}{f(\phi)^2}\;d\phi
+
C \int_{\phi_i+\delta}^{\phi_{i+1}-r(\delta)}\frac{1}{f(\phi)^2}\;d\phi.
\end{eqnarray*}
Hence, letting $\delta\to 0$ and using \eqref{Eq_05}, we obtain
\begin{eqnarray*}
&& p_{\rm st}(\phi_{i+1})-p_{\rm st}(\phi_{i})-\frac{I_i}{2\rho}\\
&=&
\lim_{\delta\to 0^+}
\left(2\rho\int_{\phi_i+\delta}^{\phi_{i+1}-r(\delta)} \frac{p_{\rm st}(\phi)}{f(\phi)^2}\;d\phi
+
C \int_{\phi_i+\delta}^{\phi_{i+1}-r(\delta)}\frac{1}{f(\phi)^2}\;d\phi\right).
\end{eqnarray*}
On the other hand, replacing $\phi=\phi_i$ and $\phi=\phi_{i+1}$ in \eqref{FK} yield that
\[
p_{\rm st}(\phi_{i})
=
p_{\rm st}(\phi_{i+1})
=
-\frac{C}{2\rho}.
\]
Consequently,
\begin{equation}\label{Eq_06}
\frac{I_i}{2\rho}
=
-\lim_{\delta\to 0^+}
\left(2\rho\int_{\phi_i+\delta}^{\phi_{i+1}-r(\delta)} \frac{p_{\rm st}(\phi)}{f(\phi)^2}\;d\phi
+
C \int_{\phi_i+\delta}^{\phi_{i+1}-r(\delta)}\frac{1}{f(\phi)^2}\;d\phi\right).
\end{equation}
By \eqref{FK_02a} we have
\[
\frac{p_{\rm st}^{'}(\phi)}{p_{\rm st}(\phi)}
=
\left(
\frac{2\rho}{f(\phi)^2}
-\frac{f^{'}(\phi)}{f(\phi)}
\right)+ \frac{C}{f(\phi)^2 p_{\rm st}(\phi)}\quad\hbox{for } \phi\in (\phi_i,\phi_{i+1}).
\]
Taking the integral of both sides from $\phi_i+\delta$ to $\phi_{i+1}-r(\delta)$ gives that
\begin{eqnarray*}
&& \ln p_{\rm st}(\phi_{i+1}-r(\delta))-\ln p_{\rm st}(\phi_{i}+\delta)\\
&=&
\int_{\phi_i+\delta}^{\phi_{i+1}-r(\delta)} \frac{2\rho}{f(\phi)^2}\;d\phi+ C \int_{\phi_i+\delta}^{\phi_{i+1}-r(\delta)} \frac{1}{f(\phi)^2 p_{\rm st}(\phi)}\;d\phi.
\end{eqnarray*}
Since
\[
\lim_{\delta\to 0^+} \ln p_{\rm st}(\phi_{i+1}-r(\delta))-\ln p_{\rm st}(\phi_{i}+\delta)
=
\ln p_{\rm st}(\phi_{i+1})-\ln p_{\rm st}(\phi_{i})=0
\]
it follows that
\[
\lim_{\delta\to 0^+}
\left(\int_{\phi_i+\delta}^{\phi_{i+1}-r(\delta)} \frac{2\rho}{f(\phi)^2}\;d\phi+ C \int_{\phi_i+\delta}^{\phi_{i+1}-r(\delta)} \frac{1}{f(\phi)^2 p_{\rm st}(\phi)}\;d\phi\right)
=0.
\]
Consequently,
\[
\lim_{\delta\to 0^+}
 \left(C\int_{\phi_i+\delta}^{\phi_{i+1}-r(\delta)} \frac{1}{f(\phi)^2}\;d\phi+ \frac{C^2}{2\rho}\int_{\phi_i+\delta}^{\phi_{i+1}-r(\delta)} \frac{1}{f(\phi)^2 p_{\rm st}(\phi)}\;d\phi\right)
=0.
\]
Adding this equality to both sides of \eqref{Eq_06} implies that
\begin{eqnarray*}
\frac{I_i}{2\rho}
&=&
-\lim_{\delta\to 0^+} \int_{\phi_i+\delta}^{\phi_{i+1}-r(\delta)} \frac{1}{f(\phi)^2}\left( 2\rho p_{\rm st}(\phi)+\frac{C^2}{2\rho p_{\rm st}(\phi)} +2 C\right)\;d\phi\\
&= &
-\lim_{\delta\to 0^+} \int_{\phi_i+\delta}^{\phi_{i+1}-r(\delta)} \frac{1}{f(\phi)^2}\left( \sqrt{2\rho  p_{\rm st}(\phi)} +\frac{C}{ \sqrt{2\rho  p_{\rm st}(\phi)}}\right)^2\;d\phi.
\end{eqnarray*}
Using the fact that the function $r$ is strictly increasing, we obtain that the function
\[
\delta\mapsto \int_{\phi_i+\delta}^{\phi_{i+1}-r(\delta)} \frac{1}{f(\phi)^2}\left( \sqrt{2\rho  p_{\rm st}(\phi)} +\frac{C}{ \sqrt{2\rho  p_{\rm st}(\phi)}}\right)^2\;d\phi
\]
is decreasing. Then, from non-constant property of the function $ p_{\rm st}(\phi)$ on $[\phi_i,\phi_{i+1}]$ (see Lemma \ref{Lemma2}(ii)) we arrive at the conclusion that $\frac{I_i}{2\rho}<0$. The proof is complete.
\end{proof}
%
%
%
%
%
%
%
%
\section*{Acknowledgements}
Authors thank to Dr. H. Nakao (Tokyo Institute of Technology) for a discussion. The initial work of this paper was done when the second author visited Hokkaido University and he thanks the support of Japan Society for the Promotion of Science. The work of Doan Thai Son and Hoang The Tuan is funded by the Vietnam National Foundation for Science and Technology Development (NAFOSTED) under Grant Number
101.03-2017.01.

\begin{thebibliography}{1}
%
\bibitem[Ar98]{Arnold}
L.~Arnold.
\newblock{\em Random Dynamical Systems.}
\newblock{Springer, 1998}.
%
\bibitem[Ba91]{Baxendale}
P.H.~Baxendale.
\newblock{Statistical equilibrium and two-point motion for a stochastic flow of diffeomorphisms.}
\newblock{\em Spatial Stochastic Processes,} (K.Alexander and J.Watkins, eds), Progress in Probability \textbf{19}, pp. 189--218, Birkh\'{a}user, Boston Basel Berlin (1991).
%
%
%
%
\bibitem[CGK01]{Kloeden}
S. Cyganowski, L. Gr\"{u}ne and P.E. Kloeden.
\newblock{Maple for stochastic differential equations.}
\newblock{\emph{Theory and Numerics
	of Differential
	Equations, Durham 2000}, pp. 127-178 (J. F. Blowey, J. P. Coleman and A.W. Craig editors)}. Springer, 2001.
%
%
\bibitem[Ku84]{Kuramoto}
Y.~Kuramoto.
\newblock{\em Chemical Oscillations, Waves, and Turbulence.}
\newblock{Springer-Verlag, Berlin, 1984.}
%
%
\bibitem[GP05a]{goldobin05}
D.S.~Goldobin and A.S~Pikovsky.
Synchronization and desynchronization of self-sustained oscillators by common noise.
 {\em Physical Review E} \textbf{71} (2005), 045201 R.
%
\bibitem[GP05b]{goldobin05-2}
D.S.~Goldobin and A.S.~Pikovsky.
Synchronization of self-sustained oscillators by common white noise.
{\em Physica A} \textbf{351} (2005), pp. 126--132.
%
%
\bibitem[GTNE10]{goldobin10}
D.S.~Goldobin. J.~Teramae, H.~Nakao and G.B.~Ermentrout. Dynamics of limit cycle oscillators subject to general noise.
{\em Physical Review Letters} \textbf{105} (2010), 154101 (1-4).
%

\bibitem[KFI12]{kurebayashi12}
W.~Kurebayashi, K.~Fujiwara and T.~Ikeguchi.
Colored noise induces synchronization of limit cycle oscillators.
{\em Europhys. Lett.} \textbf{97} (2012), 50009.
%
\bibitem[NNT05]{nagai05}
K.~Nagai, H.~Nakao and Y.~Tsubo.
Synchrony of neural oscillators induced by random telegraphic currents.
{\em Physical Review E} \textbf{ 71} (2005), 036217(1-8).
%
\bibitem[NANTK05]{nakao05}
H.~Nakao, K-s.~Arai, K.~Nagai, Y.~Tsubo and Y.~Kuramoto.
Synchrony of limit-cycle oscillators induced by random external impulses.
{\em Physical Review} \textbf{E 72} (2005), 026220(1-13).
%
\bibitem[NAK07]{nakao07}
H.~Nakao, K.~Arai and Y.~Kawamura.
Noise-induced synchronization and clustering in ensembles of uncoupled limit-cycle oscillators.
{\em Physical Review Letters} \textbf{98} (2007), 184101(1-4).
%
\bibitem[Pi84]{pikovsky84}
A. S. Pikovsky. Synchronization and stochastization of array of self-excited oscillators by external noise.
{\em Radiophysics and Quantum Electronics} \textbf{27} (1984), 390–-395.
%
%
\bibitem[TT04]{TT_04}
J.~Teramae and D.~Tanaka.
\newblock{Robustness of the noise-induced phase synchronization in a general class of lmit cycle oscillators.}
\newblock{\em Physical Review Letters} \textbf{93} (2004), no.~20, 204103.
%
\bibitem[TT06]{teramae06}
J.~Teramae and D.~Tanaka. Noise induced phase synchronization of a general class of limit cycle oscillators.
{\em Progress of Theoretical Physics Supplement} \textbf{161} (2006), 360--363.
%
%
\bibitem[TNE09]{teramae09}
J.~Teramae, H.~Nakao and G.B.~Ermentrout. Stochastic phase reduction for a general class of noisy limit cycle oscillators. {\em Physical Review Letters} \textbf{102} (2009), 194102(1-4).
%
%
%
\bibitem[YA08]{yoshimura08}
K.~Yoshimura and K. Arai. \newblock{Phase reduction of stochastic limit cycle oscillators}.
\emph{Phys. Rev. Lett.} \textbf{101} (2008), 154101.

%
%
%
\end{thebibliography}
\end{document}